\documentclass[12pt]{amsart}
\usepackage{amsfonts}
\usepackage{epstopdf}
\usepackage{amsmath, amscd, amssymb}
\usepackage{latexsym}
\usepackage{mathrsfs}
\usepackage{url}
\usepackage{caption}
\usepackage{subcaption}
\usepackage{graphicx}
\usepackage{multicol}
\usepackage{svg}

\newtheorem{theorem}{Theorem}[section]
\newtheorem{lemma}[theorem]{Lemma}
\newtheorem{corollary}[theorem]{Corollary}

\theoremstyle{remark}

\begin{document}
\title{A new unknotting operation for classical and welded knots}

\author{Danish Ali}
\address{School of Mathematical Sciences, Dalian University of Technology, China}
\email{danishali@dlut.edu.cn}

\author{Zhiqing Yang}
\address{School of Mathematical Sciences, Dalian University of Technology, China}
\email{ yangzhq@dlut.edu.cn}

\author{Mohd Ibrahim Sheikh}
\address{Department of Mathematics, Pusan National University, Republic of Korea}
\email{ibrahimsheikh@pusan.ac.kr}

\author{Sidra Batool}
\address{School of Mathematical Sciences, Dalian University of Technology, China}
\email{sidamu@mail.dlut.edu.cn}

\begin{abstract}
Any knot diagram can be transformed into the unknot by a series of unknotting operations. This paper introduces the diagonal move, a novel unknotting operation that generalizes and unifies several existing moves. We prove that the diagonal move is an efficient unknotting operation for both classical and welded knots, demonstrating that any knot or link can be reduced to the unknot or unlink via a finite sequence of diagonal moves and Reidemeister moves. Additionally, we analyze the distance between knots under diagonal moves, showing that it often requires fewer operations than traditional crossing changes, and extend our results to welded knots, confirming the diagonal move's applicability in this broader setting. Our findings provide a powerful new tool for knot simplification and equivalence, advancing topological and combinatorial knot theory.
\end{abstract}
% \clc{list here}
\maketitle

\section{Introduction}
One of the most significant challenges in knot theory is the classification of all possible knots, which involves determining which knots are equivalent and which are distinct. Another fundamental problem in knot theory is the unknotting problem which is regarded as a special case of knot classification problem. The unknotting problem can be stated as given a mathematical representation of a knot, is it possible to determine whether that knot is equivalent to the unknot. A complete and efficient classification of all knots is one of the most challenging problems that remains partially unsolved in a fully practical sense. Knots are considered equivalent if they can be continuously deformed into one another in a three-dimensional space without cutting or glueing. Equivalence between knot diagrams is determined by Reidemeister moves \cite{Reidemeister1927}. Two knot diagrams $D$ and $D'$ are equivalent if and only if $D$ can be transformed into $D'$ using Reidemeister moves ($R_1, R_2$ and $R_3$), as shown in Fig. \ref{rm}. Knot equivalence and the unknotting problem are closely linked because both involve transformations of knots and the identification of topological equivalences, but they focus on slightly different objectives. The key difference is that in the knot equivalence problem, we are comparing two arbitrary knots to see if they are equivalent, whereas in the unknotting problem, we are testing whether a single knot can be transformed into the unknot. 

\begin{figure}
\centering
\includegraphics[width=0.7\textwidth]{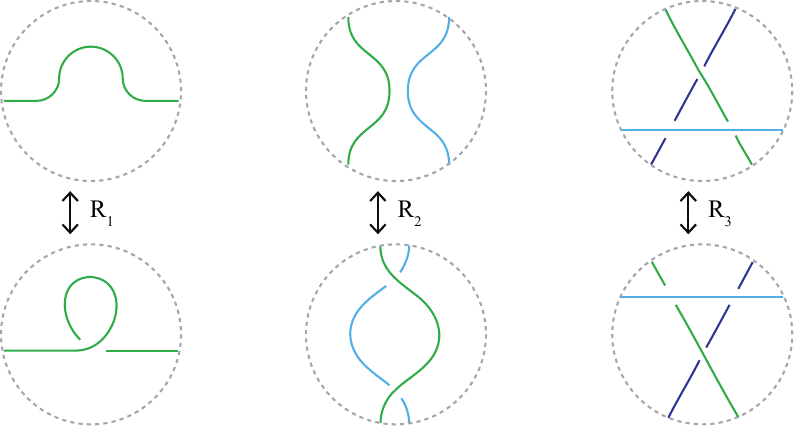}
\caption{Reidemeister moves}
\label{rm}
\end{figure}

If it is determined that a knot cannot be transformed into the unknot using only Reidemeister moves, additional techniques, known as unknotting operations, are employed to simplify the knot diagram and ultimately convert it into the unknot. An unknotting operation is a local move such that any knot diagram can be transformed into a diagram of the trivial knot by a finite sequence of these operations plus some Reidemeister moves. Unknotting operations are useful in studying knots and links. On the other hand, many knot invariants are extensively studied to distinguish different knots. Unknotting operations are inevitable in studying knot invariants. They also play an essential role in the simplification of sophisticated entangles of strings, organic compounds, and DNA in physics, chemistry, and biology. Many local moves are known as unknotting operations for knots; some well-known unknotting operations are shown in Fig.\ref{uo}. 

The crossing change is a fundamental unknotting operation for any knot diagram. To show that it is an unknotting operation we first choose a base point and an orientation for the given knot diagram. A base point is an arbitrary point on a knot diagram, distinct from crossing points. Now we travel along the knot from the base point according to the orientation of the knot; we meet every crossing two times during the complete journey. When we meet a crossing, if we first pass the crossing from the lower arc, we apply a crossing change here, if we first pass the crossing from the upper arc, we skip it. Finally, we get back to the initial point. After making all of these adjustments, the resulting diagram is called a descending diagram. A descending diagram is the unknot. This is also true in the projective space \cite{mm}. More generally, a link diagram can be unlinked with appropriate crossing changes and making each component descending as above. 

\begin{figure}
     \centering
     \begin{subfigure}{0.45\textwidth}
        \includegraphics[width=\textwidth]{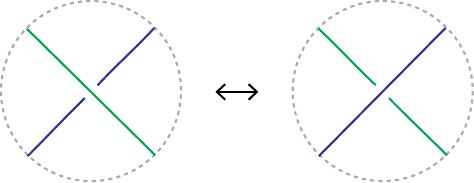}
        \caption{Crossing change}\label{cc}
        \caption*{}
     \end{subfigure}
     \hfill
        \begin{subfigure}{0.45\textwidth}
         \includegraphics[width=\textwidth]{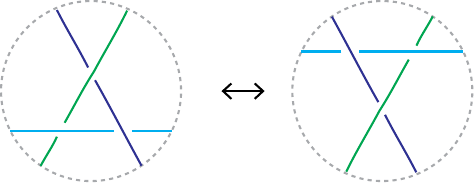}
         \caption{$\Delta$-move}\label{dm}
         \caption*{}
     \end{subfigure}
      \begin{subfigure}{0.45\textwidth}
         \includegraphics[width=\textwidth]{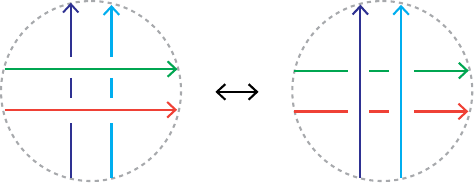}
         \caption{$\sharp$-move}\label{sm}
         \caption*{}
     \end{subfigure}
     \hfill
     \begin{subfigure}{0.45\textwidth}
         \includegraphics[width=\textwidth]{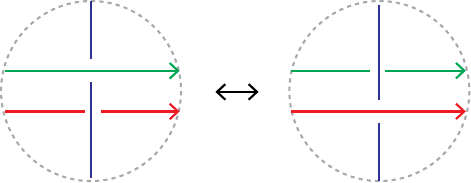}
         \caption{$\Gamma$-move}\label{pm}
         \caption*{}
     \end{subfigure}
          \begin{subfigure}{0.45\textwidth}
         \includegraphics[width=\textwidth]{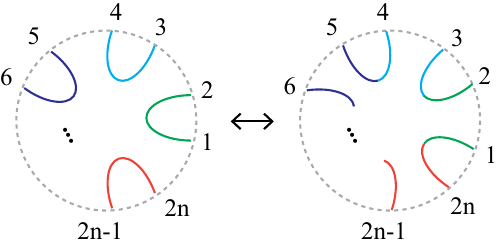}
         \caption{$H(n)$-move}  \label{hnm}
     \end{subfigure}
     \hfill
     \begin{subfigure}{0.45\textwidth}
         \includegraphics[width=\textwidth]{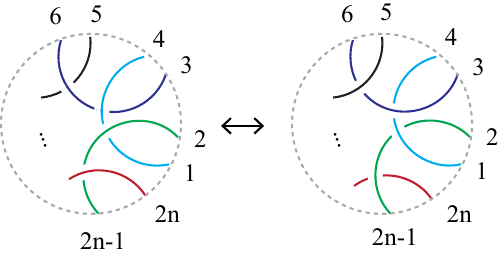}
         \caption{$n$-gon}     \label{ngm}
     \end{subfigure}
  \caption{Some well-known unknotting operations}
        \label{uo}
\end{figure}

Murakami and Nakanishi introduced the $\Delta$-move in \cite{hm}, and they proved that every knot can be deformed into a trivial knot by $\Delta$-moves. The $\sharp$-move was introduced by Murakami in \cite{hm2}; it is also proved that by some $\sharp$-moves, one can deform any knot into a trivial knot. Shibuya introduced the $\Gamma$-move and proved it is an unknotting operation \cite{ts}. The crossing change can be realized by a $\Gamma$-move. Kanenobu in \cite{kt} studied two types of $\Gamma$-moves; he showed that the two types of $\Gamma$-moves are equivalent to each other; one move is realized by the other move and vice-versa. Hoste {\it et al.} in \cite{jh} introduced the $H(n)$-move for $n \geq 2$, they proved that the $H(n)$-move is an unknotting operation. Therefore, any two knots can be transformed into each other by a finite sequence of $H(n)$-moves. The authors proved that the $H(n)$-move is an unknotting operation for virtual and welded knots and links \cite{danish2}. Nakanishi showed that a $\Delta$-move can be realized by a finite sequence of 3-gon moves \cite{yn}. So any knot can be transformed into a trivial knot by a finite sequence of 3-gon moves. Aida in \cite{ha} generalizes the notion of 3-gon moves to $n$-gon moves; it is proved that given any knot $K$, there exists an integer $n$ such that $K$ can be transformed into a trivial knot by one $n$-gon move.

\begin{figure}[ht]
     \centering
      \begin{subfigure}{0.45\textwidth}
         \includegraphics[width=\textwidth]{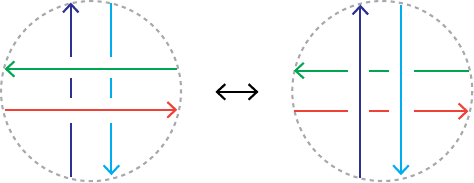}
         \caption{pass move}\label{pm}
         \end{subfigure}
     \hfill
      \begin{subfigure}{0.45\textwidth}
 \includegraphics[width=\textwidth]{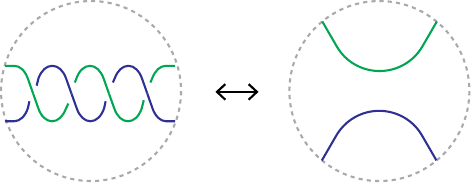}
 \caption{4-move}\label{4m}
     \end{subfigure}
    \caption{Illustration of pass move and 4-move}
        \label{uoflk}
\end{figure}

The pass move and 4-move are shown in Fig.\ref{uoflk}. We can define two knots to be pass equivalent if they are related by a finite sequence of pass moves.  In \cite{lhk} Kauffman studied the pass move and showed that any knot can be deformed into either the trivial knot or the trefoil knot by a finite sequence of pass moves. The pass move preserves the Arf invariant, two knots are pass equivalent if and only if they have the same Arf invariant. The trefoil and the unknot are not pass move equivalents, therefore, the pass move is not an unknotting operation for knots. On the other hand, in \cite{mk} Dabkowski {\it et al.} proved that all knots up to 12 crossings reduce to the trivial knot by 4-moves. They also showed that links of 2 components with at most 11 crossings reduce to either the trivial link or the Hopf link using a finite number of 4-moves. Przytycki in \cite{jhp} shows that every alternating link of two components up to 12 crossings can be reduced to the trivial link or the Hopf link by 4-moves and Reidemeister moves (See also \cite{as, ak1}). The authors study the Gordian complex of knots using 4-moves \cite{danish1}. This study provides new insights into the structure of knot diagrams and the relationships between different knot types. Whether the 4-move is an unknotting operation is an open question. 

 This paper is organized as follows. In Section \ref{idm}, we define the diagonal move and prove that the diagonal move is an unknotting operation for classical knots and links. In Section \ref{idmwk}, we discussed welded knot theory. Every welded knot diagram can be deformed into a trivial knot diagram by a finite sequence of diagonal moves plus some welded Reidemeister moves. We also studied the relationship between the diagonal move and other local moves. We showed that the crossing change, $\Delta$-move,  $\sharp$-move, $\Gamma$-move, $n$-gon move, pass move, and 4-move can be realized by a sequence of diagonal moves. We studied the distance from $K$ to $K'$ for an unknotting operation, and we discuss the relation of distance from $K$ to $K'$ for those well-known unknotting operations.

\section{Diagonal move and other local moves for classical knots and links}\label{idm}

First, we will explain the diagonal move; for simplicity, we denote it by $D$-move. Suppose that we have four arcs a, b, c and d and four crossings 1, 2, 3, and 4 as illustrated in Fig.\ref{twe}. There are two diagonal crossing pairs, \{1,3\} and \{2,4\}. A $D$-move is a local transformation on a knot or link diagram that changes only diagonal crossings as shown in  Fig.\ref{twm}. 

\begin{figure}[ht]
\centering
\includegraphics[width=3.5cm]{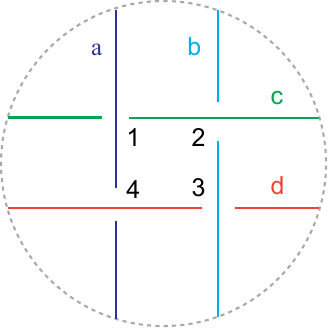}
\caption{The $D$-move illustration}
\label{twe}
\end{figure}

\begin{figure}[ht]
\centering
\includegraphics[width=13cm]{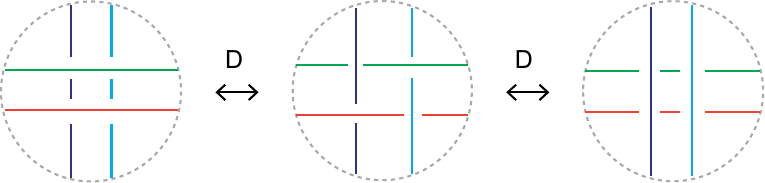}
\caption{The $D$-move}
\label{twm}
\end{figure}

\begin{theorem}
 Every knot diagram can be deformed into a trivial knot diagram by a finite sequence of $D$-moves and Reidemeister moves.
\end{theorem}

\begin{proof}
Let an oriented knot diagram $K$ has $n$ crossings $c_1, c_2, c_3 \cdots ,c_n$. We choose an arbitrary base point, distinct from crossings and endpoints of arcs. Now we start our journey from the base point along the orientation of the knot. A crossing change can be realized by a $D$-move, as shown in Fig.\ref{ccr}. When we meet a crossing for the first time, and the crossing is under crossing, we change it into over crossing by a $D$-move. When we again arrive at the same crossing for a second time, we leave it unchanged and finally, we return to the initial point. When we have done all these changes, the diagram is descending. A descending diagram is the unknot, therefore, any knot diagram can be deformed into a trivial knot diagram by a finite sequence of $D$-moves and Reidemeister moves.
\end{proof}

\begin{theorem}
 Every link diagram can be unlinked by a finite sequence of $D$-moves and Reidemeister moves.
\end{theorem}

\begin{proof}
Let $L = l_1 \cup l_2 \cup \cdots \cup l_m$ be an oriented link diagram, where each component $l_i$ is given a chosen orientation and a base point, and the components are given a fixed order $l_1, l_2, \cdots, l_m$. A crossing of $L$ is either a self-crossing, involving strands of the same component, or a mixed crossing, involving strands from two distinct components. The unlinking process begins by converting the diagram into a completely descending form. Starting with $l_1$, we traverse it once from its base point along its orientation. Whenever we meet a crossing for the first time along $l_i$, if it is a self-crossing and $l_i$ passes under at this first encounter, we perform a crossing change by a $D$-move so that $l_i$ passes over. If it is a mixed crossing with $l_j$, where $j>i$, and $l_i$ is under on the first encounter, we change it so that $l_i$ passes over. If $j<i$, no change is made because that crossing was already encountered while processing $l_j$. This procedure is then applied successively to $l_2, l_3, \dots, l_m$, ensuring that each component always passes over any later component at their first encounter.

By construction, the resulting diagram is completely descending with respect to the chosen base points, orientations, and order of components: for each component, every first encounter with a crossing occurs as an overpass, and for every mixed crossing between $l_i$ and $l_j$ with $i<j$, $l_i$ is over $l_j$. Such a diagram represents the trivial $m$-component unlink. Indeed, one can realize the diagram in $\mathbb{R}^3$ so that $l_1$ lies entirely in a small height interval above the others, $l_2$ lies just below $l_1$, and so on, respecting the over/under information. This spatial separation produces a split union of the components. Within each height slab, the projection of a component is a descending knot diagram, which is known to represent the unknot. Each component can then be isotoped to a round circle in its slab, yielding the standard unlink.
\end{proof} 

\begin{figure}
\centering
\includegraphics[width=\textwidth]{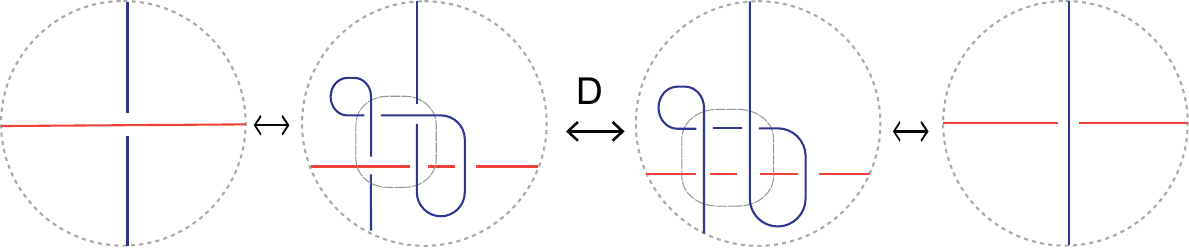}
\caption{A crossing change can be realized by a $D$-move}
\label{ccr}
\end{figure}

At first glance, Fig. \ref{ccr} suggests that $D$-move merely changes a classical crossing, however, the $D$-move is more complex than a simple crossing change. While they may involve replacing an overcrossing with an undercrossing or vice versa, they also entail additional diagrammatic modifications. Since we assert that $D$-moves constitute a new class of unknotting operations, distinct from crossing change, we must demonstrate that they cannot be reduced to a single crossing change. As illustrated in Fig. \ref{dmobdcc}, a $D$-move requires at least two crossing changes, confirming that it is not equivalent to a single crossing change. This parallels the behavior of the $\Gamma$-move, $\Delta$-move and $\sharp$-move, which are also realized as a double crossing change, triple crossing changes and quadruple crossing changes, respectively. Thus, $D$-move functions as a composite operation, combining multiple crossing modifications into a single coherent transformation.

\begin{figure}
\centering
\includegraphics[width=\textwidth]{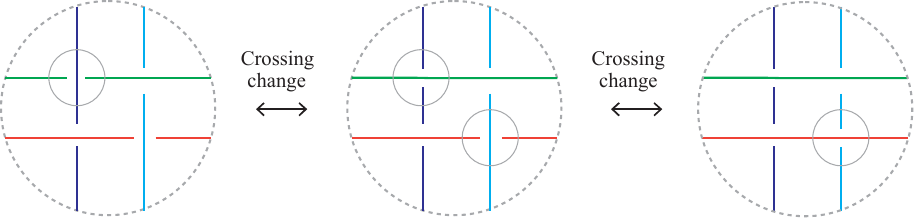}
\caption{A $D$-move is realized by at least two crossing changes}
\label{dmobdcc}
\end{figure}

To further demonstrate the enhanced capability of $D$-moves compared to single crossing changes, consider the knot diagram $K$ shown in Fig. \ref{example} (left) with nine classical crossings. While no single crossing change applied to $K$ results in the unknot, a single $D$-move and  Reidemeister moves successfully trivializes the diagram. This provides concrete evidence that $D$-moves constitute a more powerful unknotting operation that cannot be replicated by individual crossing changes alone.

\begin{figure}
\centering
\includegraphics[width=\textwidth]{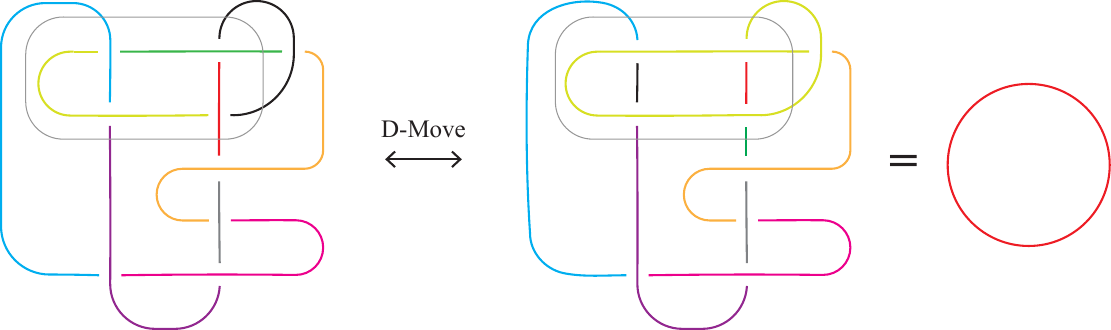}
\caption{The Knot $K$ can be unknot by a single $D$-move}
\label{example}
\end{figure}

A local move $X$ is called a generalization of another local move $Y$ if $Y$ can be realized by a finite number of $X$ local moves. We show that the crossing change, $\Delta$-move,  $\sharp$-move, $\overline{\sharp}$-move $\Gamma$-move, $n$-gon move, pass move, and 4-move can be realized by a sequence of $D$-moves therefore, we say the $D$-move is a generalization of all these moves. The crossing change, $\Gamma$-move, and 4-move can be realized by a single $D$-move, while two $D$-moves can realize the $\Delta$-move,  $\sharp$-move, and pass move. The $n$-gon move can be realized by $\leq n$ $D$-moves. All the unknotting operations shown in Fig.\ref{uo} and Fig.\ref{uoflk} can be realized by a finite sequence of $D$-moves except $H(n)$-move. The $D$-move is a combination of arcs and crossings, while the $H(n)$-move involves only arcs; therefore, the author did not find any clue to generalize $D$-move for $H(n)$-move.

We define the diagonal unknotting number $u_D(K)$ of a knot diagram $K$ to be the minimum number of $D$-moves necessary to obtain a diagram of the trivial knot from $K$. Two knots are diagonal-equivalent if one can be obtained from the other by a combination of Reidemeister moves and $D$-moves.

\begin{corollary}
Every two knot diagrams $K$ and $K'$ are diagonal-equivalent. 
\end{corollary}

In the study of oriented knot theory, there are 16 oriented versions of classical Reidemeister moves. Polyak introduced the concept of a minimal generating set for classical Reidemeister moves, and the author extends this concept to virtual knot theory \cite{ali2025}. We allow the $D$-move for all possible orientations of arcs $a, b, c$, and $d$. The $D$-move is very closely related to the $\sharp$-move and the pass move; therefore, it is essential to study the relation of $D$-move with $\sharp$-move, pass move, and other local moves.

\begin{theorem}\label{tr}
 The following moves on link diagrams can be realized by $D$-moves:
 \begin{multicols}{2}
\begin{enumerate}
\item The $\Delta$-move \item The $\sharp$-move \item The $\overline{\sharp}$-move \item The pass move \item The $\Gamma$-move \item The $n$-gon move \item The 4-move
 \end{enumerate} 
 \end{multicols}
\end{theorem}

\begin{proof}
A $\Delta$-move can be realized by a finite sequence of $D$-moves, as demonstrated in Fig.\ref{delmr}. The $D$-move allows for all possible orientations of arcs; therefore, it is elementary to show that the $\sharp$-move and pass move can be obtained by a finite sequence of $D$-moves, see Fig.\ref{smr} and Fig.\ref{pmr}. In \cite{kz} Zhang and Yang introduced a local move similar to the pass move and $\sharp$-move; they call it $\overline{\sharp}$-move. They introduced the two types of pass moves and the two types of $\overline{\sharp}$-moves. They also showed that the two types of pass moves can be obtained from each other; similarly, the two types of $\overline{\sharp}$-moves can be obtained from each other. The $\overline{\sharp}$-move is equivalent to the pass move, that is, a $\overline{\sharp}$-move and a pass move can be accomplished by a sequence of each other. Since the pass move is not an unknotting operation, so the $\overline{\sharp}$-move is not an unknotting operation for knots and links. A crossing change can be realized by a $D$-move, therefore, we can conclude that all the $\sharp$-move, pass move, and $\overline{\sharp}$-move can be obtained by a finite sequence of $D$-moves and Reidemeister moves. 

A $\Gamma$-move is realized by a finite sequence of $D$-moves, as shown in Fig.\ref{gmr}. We proved that a $D$-move can realize a crossing change; therefore, a finite sequence of $D$-moves can change all the crossing of $n$-gon move. A 4-move can be realized by a finite sequence of $D$-moves, as shown in Fig.\ref{4mr}.
\end{proof}

\begin{figure}[ht]
\centering
\includegraphics[width=13cm]{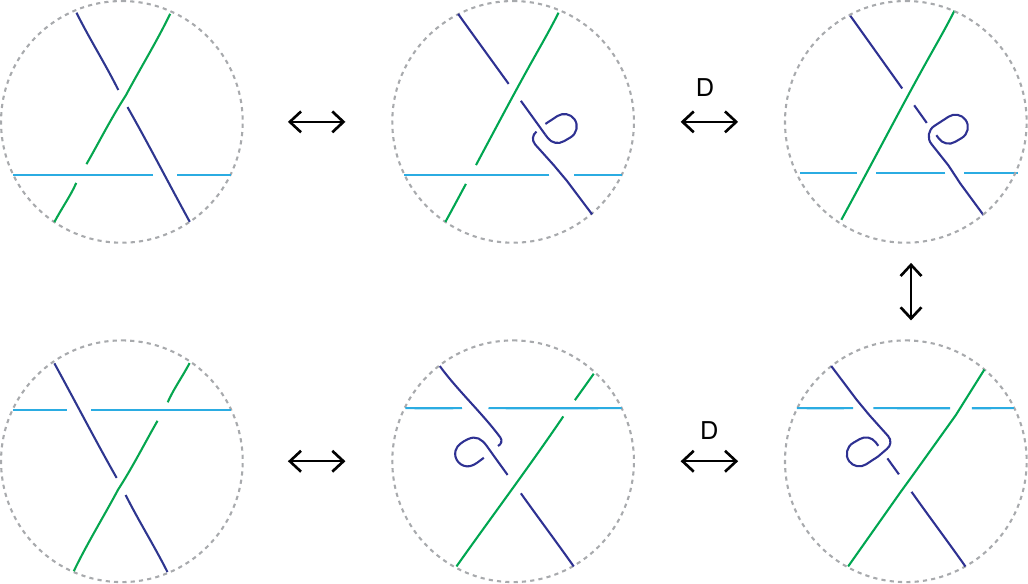}
\caption{A $\Delta$-move is realized by two $D$-moves}
\label{delmr}
\end{figure}

\begin{figure}[ht]
\centering
\includegraphics[width=13cm]{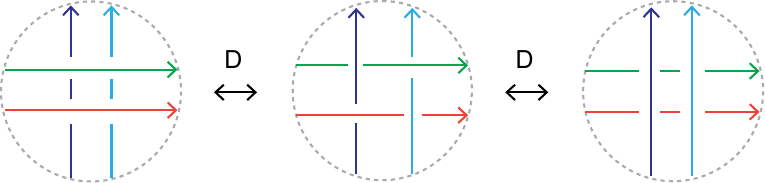}
\caption{A $\sharp$-move is realized by two $D$-moves}
\label{smr}
\end{figure}

\begin{figure}[ht]
\centering
\includegraphics[width=13cm]{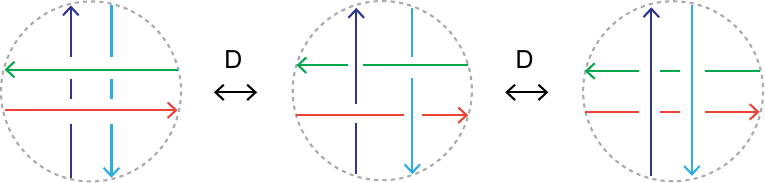}
\caption{A pass move is realized by two $D$-moves}
\label{pmr}
\end{figure}

\begin{figure}[ht]
\centering
\includegraphics[width=\textwidth]{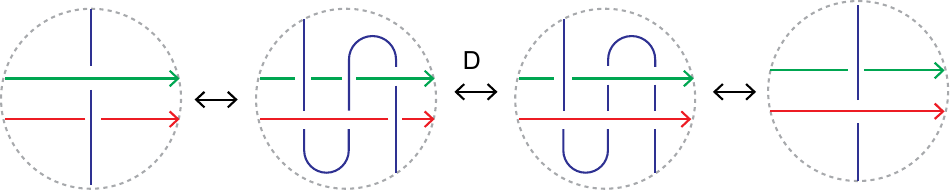}
\caption{A $\Gamma$-move is realized by a $D$-move}
\label{gmr}
\end{figure}

\begin{figure}[ht]
\centering
\includegraphics[width=\textwidth]{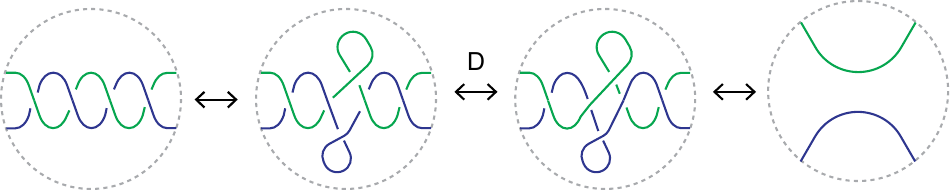}
\caption{A 4-move is realized by a $D$-move}
\label{4mr}
\end{figure}

\section{diagonal unknotting operations for welded knots and links}\label{idmwk}
Welded knot theory was introduced in \cite{rf} as a generalization of a classical knot theory. A welded knot diagram is a knot diagram that may have welded crossings as well as classical crossings. A welded knot is an equivalence class of welded knot diagrams under the three kinds of classical Reidemeister moves together with the five types of welded Reidemeister moves (for more details see \cite{ss}). Similar to classical knots, invariants and local moves also play important roles in welded knot theory. Some unknotting operations of classical knots are extended to welded knots. For example, Shin Satoh proved that the crossing changes, $\Delta$-moves, and $\sharp$-moves are unknotting operations for welded knots \cite{ss}. In classical knot theory, it is known that any classical knot can be deformed into the trivial knot or the trefoil knot by a finite sequence of pass moves. The trefoil and the unknot are not pass move equivalents, therefore, the pass move is not an unknotting operation
for classical knots. However, in \cite{tn} it is proved that the pass move is an unknotting operation for welded knots. If we combine the results of \cite{ss} and \cite{tn} then we have the following theorem.

\begin{theorem}\label{weldthe}
The following local moves are unknotting operations for welded knots:
\begin{multicols}{2}
\begin{enumerate}
\item The crossing change \item The $\Delta$-move \item The $\sharp$-move \item The pass move
 \end{enumerate} 
 \end{multicols}
\end{theorem}
An orientated and based welded knot diagram $W$ is called a descending diagram if walking along $W$ from the base point following the orientation, we meet the over-crossing first and the under-crossing later at every classical crossing. Any descending welded knot diagram $W$ is related to the trivial diagram by a finite sequence of welded Reidemeister moves. The crossing change, $\Delta$-move, and $\sharp$-move are unknotting operations for both classical and welded knots. The crossing change, $\Delta$-move, $\sharp$-move, and pass move can be realized by a finite sequence of $D$-moves. Therefor, it is obvious that a $D$-move is an unknotting operation for welded knots. 

\begin{corollary}
Every welded knot diagram can be deformed into a trivial knot diagram by a finite sequence of $D$-moves.
\end{corollary}

So the $D$-move is an unknotting operation for welded knots, and we have the following corollary.

\begin{corollary}
Every two welded knot diagrams $W$ and $W'$ are $D$-equivalent. 
\end{corollary}

It is a natural question to ask how far apart two distinct knots are. Distance on knots is used to answer this question. Unknotting operations play an important role in defining distance on knots. Given an unknotting operation on two distinct knots, $K$ and $K’$, a diagram of $K$ can be transformed into a diagram of $K'$ by a finite sequence of this unknotting operation plus some Reidemeister moves. For two knots $K$ and $K'$, the distance from $K$ to $K'$ for an unknotting operation is the minimum number of times this unknotting operation must be used to transform a diagram of $K$ into a diagram of $K'$, where the minimum is taken over all diagrams of $K$ and $K'$. If a knot diagram of $K$ is transformed into the unknot by a finite sequence of this unknotting operation, then the distance is called the unknotting number of $K$. The distance on knots defines a metric on the space of knots. 

We define the diagonal unknotting number $u_D(K)$ of a knot diagram $K$ to be the minimum number of $D$-moves necessary to obtain a diagram of the trivial knot from $K$. Two knots are $D$-equivalent if one can be obtained from the other by a combination of Reidemeister moves and $D$-moves.

\begin{corollary}
Every two knot diagrams $K$ and $K'$ are $D$-equivalent. 
\end{corollary}

The crossing change, $\Delta$-move, $\sharp$-move,  $\Gamma$-move, $n$-gon move, and $D$-move are unknotting operations for both classical and welded knots and links. For two knots $K$ and $K'$, the distance from $K$ to $K'$ defined by the crossing change, $\Delta$-move, $\sharp$-move,  $\Gamma$-move, and $D$-move are denoted by $d_{X}(K,K')$, $d_{\Delta}(K,K')$, $d_{\sharp}(K,K')$, $d_{\Gamma}(K,K')$, and $d_{D}(K,K')$ respectively. In Theorem \ref{tr} we proved that the crossing change, $\Gamma$-move and 4-move can be realized by a single $D$-move, while two $D$-moves can realize the $\Delta$-move,  $\sharp$-move, and pass move. Therefor, we have the following theorem.

\begin{theorem}\label{thd1}
For any two knots $K$ and $K'$ $$d_{D}(K,K') \leq d_{X}(K,K') \leq  2d_{D}(K,K') $$ 
\end{theorem}
\begin{proof}
A crossing change can be realized by only one $D$-move, as shown in Fig.\ref{ccr}. Since a $D$-move can be replaced by two crossing changes, therefore, we conclude that $d_{D}(K,K') \leq d_{X}(K,K') \leq  2d_{D}(K,K') $.
\end{proof}

\begin{lemma}
$$d_{D}(K,K') \leq d_{\Gamma}(K,K')$$
\end{lemma}
\begin{proof}
A $\Gamma$-move can be realized by only one $D$-move, as shown in Fig.\ref{gmr}. However, a $\Gamma$-move can be realized by two crossing changes.  Therefore, $d_{D}(K,K') \leq d_{\Gamma}(K,K')$ .
\end{proof}

\begin{lemma}
$$2d_{D}(K,K') \leq 2d_{X}(K,K')  \leq d_{\Delta}(K,K')$$ 
\end{lemma}

\begin{proof}
A $\Delta$-move can be realized by two crossing changes. Furthermore, Theorem \ref{thd1} proves $d_{D}(K,K') \leq d_{X}(K,K')$ . So we have $2d_{D}(K,K') \leq 2d_{X}(K,K')  \leq d_{\Delta}(K,K')$.
\end{proof}

\begin{lemma}
$$ 2d_{D}(K,K') \leq d_{\sharp}(K,K') $$
\end{lemma}
\begin{proof}
A $\sharp$-move can be realized by two $D$-moves, as shown in Fig.\ref{smr}. Therefore,$ 2d_{D}(K,K') \leq d_{\sharp}(K,K') $.
\end{proof}

\noindent {\bf Examples}
\begin{enumerate}
    \item Suppose  $9_1$ and $5_1$ are the standard knot diagrams in \cite{ki}, then $d_{D}(9_1,5_1)=1 \leq d_{X}(9_1,5_1)=2$.
    
    \item Suppose $5_1$ is the standard knot diagram and $0$ is the unknot, then $d_{D}(5_1,0)=1 \leq d_{X}(5_1,0)=2$. In this example, $0$ is the unknot, therefore the distance is also known as unknotting number, therefore we have $u_{D}(5_1,0)=1 \leq u_{X}(5_1,0)=2$
    
    \item From example 1 and example 2 we can obtain the following relation $u_{D}(9_1,0)=2 \leq u_{X}(9_1,0)=4$.
\end{enumerate}

% \section*{Declarations}
% \subsection*{Conflict of interest} The authors declare that they have no conflict of interest. The contents of this manuscript have not been published previously.

\bibliographystyle{amsplain}

\end{document}